\documentclass[12pt,hyp,]{amsart}

\usepackage{rotating}
\usepackage{tikz}
\usepackage{tikz-cd}
\usepackage{subfig}
\usepackage{adjustbox}
\usepackage{cite}
\usepackage[colorlinks, linkcolor=blue, allcolors=blue]{hyperref}
\hypersetup{nesting=true,debug=true,naturalnames=true}
\usepackage[margin=1in]{geometry}

\usetikzlibrary{decorations.markings}

\title{Rank inequalities on knot Floer homology of periodic knots} 

\author{Keegan Boyle}  

\email{kboyle@math.ubc.ca}  

\subjclass[2010]{}

\newcommand{\Z}{{\mathbb{Z}}}

\newcommand{\F}{{\mathbb{F}}}

\newtheorem{lemma}{Lemma}
\newtheorem{proposition}{Proposition}

\newtheorem{theorem}{Theorem}
\newtheorem{conjecture}{Conjecture}

\theoremstyle{definition}
\newtheorem{definition}{Definition}[section]
\newtheorem{remark}[definition]{Remark}
\newtheorem{example}[definition]{Example}

\newtheorem*{thm:hls}{Theorem \ref{thm:hls}}
\newtheorem*{conj:ineq}{Conjecture \ref{conj:ineq}}

\DeclareMathOperator{\rank}{rank}

\begin{document}
\begin{abstract}
Let $\widetilde{K}$ be a 2-periodic knot in $S^3$ with quotient $K$. We prove a rank inequality between the knot Floer homology of $\widetilde{K}$ and the knot Floer homology of $K$ using a spectral sequence of Hendricks, Lipshitz and Sarkar. We also conjecture a filtered refinement of this inequality, for which we give computational evidence, and produce applications to the Alexander polynomials of $\widetilde{K}$ and $K$. 
\end{abstract}
\maketitle
%\tableofcontents

\section{Introduction}
A \emph{$p$-periodic knot} $\widetilde{K} \subset S^3$ is one which is fixed by a $\Z/p\Z$ action on $S^3$ such that the fixed set (or \emph{axis}) $\widetilde{A}$ of the action is an unknot disjoint from $\widetilde{K}$. We refer to the image of this knot in the quotient $S^3/(\Z/p\Z) \cong S^3$ as the \emph{quotient knot} $K$, and the image of the axis $\widetilde{A}$ as the axis $A$. Periodic knots have been studied extensively, and although hyperbolic geometry and other tools can often determine the periods and quotients of a particular knot, many relations between periodic knots and knot invariants are unknown. Useful obstructions to these questions come from Murasugi \cite{Mu}, who proved that the Alexander polynomial of the quotient knot divides the Alexander polynomial of the periodic knot, and Edmonds \cite{E}, who proved an inequality involving the genus of the periodic knot and the genus of the quotient. 

A potential newer tool to study these questions is the knot invariant knot Floer homology, developed by Ozsv\'ath and Szab\'o \cite{OS5} and independently Rasmussen \cite{R}. Knot Floer homology is a bigraded abelian group $\widehat{\mathit{HFK}}_i(K,a)$, which is defined using techniques from symplectic geometry. This invariant categorifies the Alexander polynomial in the sense that the Alexander polynomial is the Euler characteristic of $\widehat{\mathit{HFK}}$ \cite{OS4}. In light of the classical results about the Alexander polynomials of periodic knots, it is natural to ask what information $\widehat{\mathit{HFK}}$ carries about periodic knots.

Some work has already been done in the direction of understanding the relationship between periodic knots and knot Floer homology. Using a localization theorem of Seidel and Smith \cite{SS}, Hendricks \cite{H} constructed a spectral sequence from $\widehat{\mathit{HFK}}(\widetilde{K})$ to $\widehat{\mathit{HFK}}(K)$ for $2$-periodic knots $\widetilde{K}$. This spectral sequence was later refined by Hendricks, Lipshitz, and Sarkar \cite{HLS}. The main result of this paper is a corollary of the spectral sequence \cite[Theorem 1.16]{HLS}, stated in Theorem \ref{thm:hls}. In Conjecture \ref{conj:ineq}, we further conjecture a refinement of Theorem \ref{thm:hls} by filtering along the homological grading. Theorem \ref{thm:hls} and Conjecture \ref{conj:ineq} each imply new information about the Alexander polynomials of periodic knots.
\begin{theorem}
 \label{thm:hls}
 Let $\widetilde{K}$ be a $2$-periodic knot in $S^3$ with quotient knot $K$. Let $\lambda$ be the linking number of the axis with $K$. Then there is a rank inequality
\[
\sum_i \rank \bigg{(} \widehat{\mathit{HFK}}_i(\widetilde{K}, 2a + \dfrac{\lambda -1}{2}) \oplus  \widehat{\mathit{HFK}}_i(\widetilde{K}, 2a + \dfrac{\lambda +1}{2}) \bigg{)}\geq \sum_i \rank \widehat{\mathit{HFK}}_i(K, a)
\]
for all $i,a \in \Z$.
 \end{theorem}
The following conjecture proposes a version of the rank inequality in Theorem \ref{thm:hls} filtered along the Maslov (or homological) grading $i$.
\begin{conjecture} 
Let $\widetilde{K}$ be a $2$-periodic knot in $S^3$ with quotient knot $K$ and axis $A$, and let $\lambda$ be lk$(K,A)$. Then
\[
\sum_{i \geq q} \rank \bigg{(}  \widehat{\mathit{HFK}}_i(\widetilde{K},\widetilde{a}) \oplus \widehat{\mathit{HFK}}_i(\widetilde{K},\widetilde{a} + 1) \bigg{)} \geq  \sum_{2i \geq q+1} \rank \widehat{\mathit{HFK}}_i(K,a)
\]
and
\[
\sum_{i \leq q} \rank \bigg{(} \widehat{\mathit{HFK}}_i(\widetilde{K},-\widetilde{a}) \oplus \widehat{\mathit{HFK}}_i(\widetilde{K},-\widetilde{a} - 1) \bigg{)} \geq  \sum_{2i \leq q-1} \rank \widehat{\mathit{HFK}}_i(K,-a),
\]
where $\widetilde{a} = 2a + \dfrac{\lambda-1}{2}$.
\label{conj:ineq}
\end{conjecture}
By standard properties of knot Floer homology (see section \ref{sec:3}), the second inequality would follow from the first by considering the mirrors of $K$ and $\widetilde{K}$.

\subsection{Organization}
In Section \ref{sec:morse} we lay out the motivation for Theorem \ref{thm:hls} and Conjecture \ref{conj:ineq}, and prove the corresponding statements in Morse homology. In Section \ref{sec:3} we prove Theorem \ref{thm:hls}, and review some important properties of knot Floer homology which will be useful in Section \ref{sec:4}. In Section \ref{sec:4} we prove applications of Theorem \ref{thm:hls} and Conjecture \ref{conj:ineq} to the Alexander polynomial. Finally, in Section \ref{sec:5} we provide computational and theoretical evidence for Conjecture \ref{conj:ineq}, and explain where the proof in Section \ref{sec:morse} breaks down when applied to knot Floer homology.

\subsection{Acknowledgements}
The author would like to thank Kristen Hendricks and Robert Lipshitz for helpful conversations.

\section{Motivation from Morse homology}\label{sec:morse}

Floer homology theories are modeled on Morse homology, and Theorem \ref{thm:hls} and Conjecture \ref{conj:ineq} are Floer-theoretic analogs of rank inequalities in Morse homology. Specifically, Theorem \ref{thm:hls} is an analog of the following classical result of Smith theory, first developed by Smith \cite{Smith38, Smith39, Smith41}.
\begin{theorem}
Let $X$ be finite-dimensional $G$-CW complex for a finite order $p$-group $G$, with fixed set $F$. Then
\[
\sum_{i \in \Z} \rank H_i(X;\F_p) \geq \sum_{i \in \Z} \rank H_i(F;\F_p).
\]
\end{theorem}
A first attempt at refining this statement might be to restrict the inequality to each homological grading. However, this is immediately false. Consider the case that $X = S^2$, and $G = \Z/2\Z$ acts by reflection so that $F = S^1$. Then $H_1(S^2;\F_2) = 0$, but $H_1(S^1;\F_2) \neq 0$.

However, with more care two refinements to this inequality have been shown. One is the following theorem of Floyd, which is our model for Conjecture \ref{conj:ineq}. Another was proved more recently in \cite{M}. We have also included a modern proof of Floyd's theorem here in the hope that it may be adapted to the knot Floer homology case. See Section \ref{sec:5.3} for further discussion.
\begin{theorem}{\cite[Theorem 4.4]{F}}
\label{thm:floyd}
Let $X$ be a locally compact finite dimensional Hausdorff space. Let $\tau$ be a periodic map on $X$ of prime period $p$, and let $F$ be the fixed set of $\tau$. Then for all $n \in \Z$
\[
\sum_{i \geq n} \rank H_i(X;\F_p) \geq \sum_{i \geq n} \rank H_i(F;\F_p).
\]
\end{theorem}

Floyd's original proof of this fact uses certain long exact sequences in homology. However, in the case where $X$ is a $\Z/p\Z$-CW complex, we can reprove this statement using a spectral sequence similar to (\ref{eqn:ss}). We will restrict to the case $p=2$ for simplicity. The key step in the proof which does not immediately generalize to the knot Floer homology case is the following lemma.
\begin{lemma}
Let $C_*(X)$ be the complex of cellular chains on $X$. Then the subspace of $C_*(X)$ generated by fixed cells is a subcomplex $C_*^{\text{fix}}(X)$.
\label{lemma:sub}
\end{lemma}
\begin{proof}
By the definition of a $G$-CW complex, if a cell has a fixed point then the entire cell is fixed, and by continuity of the group action if a cell is fixed then so is its boundary.
\end{proof}

To see how Theorem \ref{thm:floyd} follows from Lemma \ref{lemma:sub}, consider the following bicomplex of cellular chains on $X$.
\[
\begin{tikzcd}
\dots & C_*(X)  \arrow{l}[above]{1 + \tau}  & C_*(X)  \arrow{l}[above]{1 + \tau} & \dots  \arrow{l}[above]{1 + \tau}
\end{tikzcd}
\]
Consider the spectral sequence $^{h}E_{p,q}^r$ coming from taking the horizontal differentials first. 
\begin{lemma}
The spectral sequence $^{h}E_{p,q}^r$ converges to 
\[
H_{i}(F) \otimes \F_2[u, u^{-1}] \cong  \bigoplus_{p+q = i}\ ^{h}E_{p,q}^{\infty}.
\]
\end{lemma}
\begin{proof}
The $E^1$ page is $C_{i}(F) \otimes \F_2[u, u^{-1}]$, since the cells in the kernel mod image of $1 + \tau$ are exactly those fixed by $\tau$. Then the differential on the $E^2$ page is precisely the differential in $C_i(F)$, and all further differentials are 0. Indeed, a non-zero differential on a subsequent page would include a non-zero map from a fixed cell to a non-fixed cell, contradicting Lemma \ref{lemma:sub}.
\end{proof}
On the other hand we also have a spectral sequence $^{v}E_{p,q}^r$ from taking the vertical differentials first. This spectral sequence has 
\[
H_{i}(X) \otimes \F_2[u, u^{-1}] \cong \bigoplus_{p+q = i}\ ^{v}E_{p,q}^1.
\]
However, this spectral sequence must converge to the same homology as $^{h}E_{p,q}^r$ since $X$ is finite-dimensional and hence has a bounded cellular chain complex. Hence we get a spectral sequence from $H_{i}(X) \otimes \F_2[u, u^{-1}]$ to $H_{i}(F) \otimes \F_2[u, u^{-1}]$. This implies the classical Smith inequality
\[
|H_*(X;\F_2)| \geq |H_*(F;\F_2)|
\]
where $|H_*(X;\F_2)|$ is the total dimension of $H_*(X;\F_2)$.

We would like to refine this result to be filtered by the vertical grading in the spectral sequence. To do so, we will need the following definitions and lemma, which apply more generally to any bicomplex of $\F_2$-vector spaces. In this setting we will refer to the horizontal differential as $\partial_h$ and the vertical differential as $\partial_v$.

\begin{definition}
A \emph{square} is any bicomplex of $\F_2$-vector spaces consisting of four non-zero generators $a,b,c,$ and $d$ with $\partial_h(b) = a, \partial_h(d) = c, \partial_v(a)=c$, and $\partial_v(b)=d$. Graphically, we indicate this as shown in the left part of Figure \ref{fig:summands}.

Similarly, a \emph{staircase} is any bicomplex of $\F_2$-vector spaces as shown in the right part of Figure \ref{fig:summands}. Algebraically, a staircase is a collection of generators $\{a_i,b_i|0\leq i\leq n\}$ with $\partial_h(b_i) = a_i$ and $\partial_v(b_i) = a_{i+1}$, where $a_0$ or $b_n$ may be 0, but all other $a_i$ and $b_i$ are non-zero.

The \emph{length} of a staircase is the number of isomorphisms $\partial_h(b_i) = a_i$ and $\partial_v(b_i) = a_{i+1}$ in the diagram, so that a staircase of length 0 is a single generator, and a staircase of length 1 is a single isomorphism between generators.
\end{definition}
\begin{figure}

\begin{tikzcd}
a \arrow[d] & b \arrow[d] \arrow[l] \\
c & d \arrow[l]\\
\end{tikzcd}  
\qquad
\begin{tikzcd}
\dots &b_i \arrow[l] \arrow[d] & \\
&a_{i+1} & b_{i+1} \arrow[l] \arrow[d] \\
&&\dots \\
\end{tikzcd}  

\caption{Square (left) and staircase (right) bicomplexes.}
\label{fig:summands}
\end{figure}

This terminology allows us to break apart bicomplexes into understandable pieces. The following proposition is also proved in \cite{Stelzig}.

\begin{proposition} \cite{K}
Vertically bounded bicomplexes of $\F_2$-vector spaces decompose as direct sums of staircases and squares. 
\label{lemma:stair}
\end{proposition}

To prove this proposition, we first give the following lemmas and definition.

\begin{lemma}
Every square subcomplex of a bicomplex of $\F_2$-vector spaces is a direct summand.
\label{lemma:square}
\end{lemma}
\begin{proof}
Any bicomplex $C$ of $\F_2$-vector spaces is a module over $\F_2[x,y]/(x^2,y^2)$, which is a Frobenius algebra. Any square is a rank 1 free module, and hence projective. However, projective modules over a Frobenius algebra are injective as well, and hence summands.
\end{proof}

\begin{definition}
Let $S= \{a_i, b_i|0 \leq i \leq n\}$ and $S' = \{a_i',b_i'|0 \leq i \leq m\}$ be a pair of disjoint staircase summands of a bicomplex $C$ such that the bigrading of $a_0$ is the same as the bigrading of $a_0'$, the length of $S$ is less than or equal to the length of $S'$, and either $a_0, a_0' \neq 0$ or $a_0 = a_0' = 0$. Then $S$ and $S'$ occupy the same diagonal, $S$ is not longer than $S'$, and they begin in the same bigrading. Let \emph{the sum of $S$ and $S'$}, $S + S'$, be the staircase $\{a_i + a_i', b_i + b_i'| 0 \leq i \leq n\} \cup \{a_i',b_i'|n < i \leq m\}$ if $b_n \neq 0$, or $\{a_i + a_i', b_i + b_i'| 0 \leq i \leq n\}$ if $b_n = 0$. 
\end{definition}

\begin{figure}
\begin{center}
\begin{tikzcd}
a_0 &b_0 \arrow[l] \arrow[d] & \\
&a_1 & b_1 \arrow[l] \arrow[d] \\
&&a_2 \\
\end{tikzcd}  
\qquad $\oplus$
\begin{tikzcd}
a_0' &b_0' \arrow[l]& \\
&&\\
&&\\
\end{tikzcd}  
\qquad $=$
\begin{tikzcd}
(a_0 + a_0') &(b_0+b_0') \arrow[l] \arrow[d] & \\
&a_1 & b_1' \arrow[l] \arrow[d] \\
&&a_2 \\
\end{tikzcd}  
\qquad $\oplus$
\begin{tikzcd}
a_0' &b_0' \arrow[l]& \\
&&\\
&&\\
\end{tikzcd}  
\end{center}
\caption{Replacing a staircase with the sum of two staircases in a direct sum.}
\label{fig:squaresum}
\end{figure}

\begin{lemma}
In the notation above, $S+S'$ is a summand of $C$. Furthermore, if $b_n =0$ then $S \oplus S' = (S+S') \oplus S'$, and if $b_n \neq 0$, then $S \oplus S' = (S + S') \oplus S$. In particular, we can replace one of $S$ or $S'$ with $S + S'$ in a staircase decomposition of $C$. See Figure \ref{fig:squaresum} for an example.
\label{lemma:stairsum}
\end{lemma}
\begin{proof}
It is clear that if $b_n = 0$, then $S \cup S' = S' \cup (S+S')$ and if $b_n \neq 0$, then $S \cup S' = S \cup (S+S')$, and so to check that $S + S'$ is a summand of $C$ it is enough to check that it is a summand of $S \cup S'$. The only place where this might fail is at $a_n$ or $b_n$ where $S$ ends. 

First, if $b_n = 0$, then $S$ ends in a vertical differential, and indeed the final element $a_{n-1} + a_{n-1}'$ of $S + S'$ is not in the image of $\partial_h$ since $a_{n-1}$ is not but $a_{n-1}'$ is. 

Second, if $b_n \neq 0$ then $S$ ends in a horizontal differential, and applying the vertical differential to the element $b_n + b_n'$ of $S+S'$ gives exactly $b_{n+1}'$ so that $S+S'$ is a staircase, as desired.
\end{proof}

\begin{proof}[Proof of Proposition \ref{lemma:stair}] 
Let $C$ be a vertically bounded bicomplex of $\F_2$ vector spaces which has finite dimension in each bigrading and with horizontal differential $\partial_h$ and vertical differential $\partial_v$. By Lemma \ref{lemma:square}, any square subcomplex of $C$ is a summand, so we can quotient these out to get a new bicomplex without any square subcomplexes. We therefore assume there are no square subcomplexes, and in particular, no compositions of horizontal and vertical isomorphisms
\[
\begin{tikzcd}
y \arrow[d, "\partial_v"]
& x \arrow[l, "\partial_h"] \\
z 
\end{tikzcd}  
\ \ \ \ \ \text{or}\ \ \ \ \ 
\begin{tikzcd}
 & x \arrow[d, "\partial_v"] \\
z & y \arrow[l, "\partial_h"] 
\end{tikzcd}  ,
\]
since either of these would necessarily complete to a square by commutativity of the bicomplex. That is, $\partial_h \circ \partial_v = \partial_v \circ \partial_h = 0$.

We now claim that there is a choice of basis which splits $C$ into a direct sum of staircases. We will prove this claim by induction on the number of non-trivial vertical degrees. 

For the base case, we have a single horizontal chain complex. We first choose a basis for the image of $\partial_h$, then extend it to a basis for the kernel of $\partial_h$. Now we choose preimages of the kernel basis elements where possible, and use these elements to extend the basis to the entire complex. By construction this decomposes our complex into trivial staircases (basis elements which are in the kernel but not the image of $\partial_h$),  and length 1 staircases (isomorphisms between the 1-dimensional subspaces spanned by basis elements given by $\partial_h$). Furthermore, each of these is a summand. 

Now consider a bicomplex $C'$ with bounded vertical degrees, which by the inductive assumption has a basis which decomposes it into staircase summands. We will add a horizontal chain complex $C_{top}$ in a new top vertical degree to get a complex $C = C_{top} \to C'$, and we will construct a staircase summand in $C$ which begins in an arbitrary grading of $C_{top}$.

We consider two cases. To begin, suppose $(C_{top},\partial_h)$ is not exact, and choose a basis for $C_{top}$ splitting it into staircase summands as in the base case, and choose a basis element $a \in C_{top}$ which is in the kernel but not the image of $\partial_h$. We will construct a staircase summand containing $a$. Any staircase containing $a$ must start at $a$ since $\partial_h(a) = 0$, so it remains to consider $\partial_v(a)$. 

In this direction, write $\partial_v(a) = b_1 + b_2 + \dots + b_n$ for some basis elements $b_i$ in $C'$. Each $b_i$ is contained in a unique staircase summand in $C'$ by the inductive assumption, and since $\partial_h \circ \partial_v = 0$, $\partial_h(b_i) = 0$ so that these staircases all start in the same bigrading. Now by Lemma \ref{lemma:stairsum} and induction we can find a change of basis for $C'$ decomposing it into new staircase summands (of $C'$) so that $b_1 + b_2 + \dots + b_n = b_1'$ is a basis element and hence contained in one of the staircases. 

We then have a staircase in $C$, but it may be the case that for some other basis elements $c$ in $C_{top}$ and $\{b_i'\}$ in $C'$, $\partial_v(c) = b_1' + b_2' + \dots + b_m'$. That is, it is not obvious in this basis that our staircase is a summand. To fix this, we will change the basis of $C_{top}$ by replacing $c$ with $c+a$, and repeat as necessary until for any basis element $c$ in $C_{top}$ the image under $\partial_v$ is a sum of basis elements disjoint from $\partial_v(a)$. Since $\partial_v \circ \partial_h = 0$, none of these new basis elements are in the image of $\partial_h$, so that in the new basis the staircase is clearly a summand of $C$.

Alternatively, suppose that $(C_{top},\partial_h)$ is exact. Then choose a basis element $a$ in $C_{top}$ which is not in the image of $\partial_h$, and apply basis changes as above. It then remains to consider $\partial_h(a)$. By construction this is a basis element, and so our staircase will be a summand unless $\partial_h(c) = \partial_h(a)$ for some other basis element $c$ in $C_{top}$. In this case replace $c$ with $c+a$ in the basis, and since $C_{top}$ is exact, $c+a$ is in the image of $\partial_h$ and hence in the kernel of $\partial_v$ since $\partial_v \circ \partial_h = 0$. In particular, this change will preserve the condition that $\partial_v(c)$ is a disjoint set of basis elements from $\partial_v(a)$. 

Now for any complex $C_{top}$ we have constructed a staircase summand for $C$, and hence by induction we can decompose $C$ into staircase summands since $C$ is finite dimensional in each bigrading. 
\end{proof}
We now return to the bicomplex of cellular chains on $X$, and give a final lemma before completing the proof of Theorem \ref{thm:floyd}.
\begin{lemma}
There exists a decomposition of the bicomplex 
\[
\begin{tikzcd}
\dots & C_*(X)  \arrow{l}[above]{1 + \tau}  & C_*(X)  \arrow{l}[above]{1 + \tau} & \dots  \arrow{l}[above]{1 + \tau}
\end{tikzcd}
\]
as in Proposition \ref{lemma:stair} such that each staircase with $a_0 = 0$ is length 1.
\label{lemma:top}
\end{lemma}
\begin{proof}
Start with the decomposition into summands given from Proposition \ref{lemma:stair}. Consider a summand consisting of a single staircase of length greater than 1, and for which $a_0=0$. That is, a staircase which begins with a vertical isomorphism $d(b_0) = b_1 + \tau b_1$. Then observe that $b_0 + \tau b_0 = 0$, and hence $b_0$ is fixed by $\tau$. Now we can write $b_0 = \alpha + \beta + \tau \beta$ where $\alpha \in C_*^{\text{fix}}(X)$ and $\beta$ is in the subspace consisting of generators which are not fixed by $\tau$. 

Since $C_*^{\text{fix}}(X)$ is a subcomplex by Lemma \ref{lemma:sub}, $d(\alpha) = 0$. This implies that $d(\beta + \tau \beta) = d(b_0) = b_1 + \tau b_1$, and hence that $\beta + \tau \beta \overset{d}{\to} b_1 + \tau b_1$ was not part of a square summand. In particular, we have
\[
\begin{tikzcd}
\beta + \tau \beta \arrow[d, "d"]
& \beta \arrow[l, "1 + \tau"] \\
 b_1 + \tau b_1
\end{tikzcd}  
\]
as part of our summand, which is a contradiction with Proposition \ref{lemma:stair}.
\end{proof}

\begin{proof}[Proof of Theorem \ref{thm:floyd} in the case $p=2$ and $X$ is an $\Z/p\Z$-CW complex]
Combining Proposition \ref{lemma:stair} and Lemma \ref{lemma:top}, we see that all generators of $^{v}E^{\infty}_{p,q}$ are represented by staircases in the bicomplex with $a_0 \neq 0$ and $b_n=0$. That is, staircases which end with a horizontal arrow on the top, and a vertical arrow on the bottom. 

Now for any generator of $H_*(F)$, consider the staircase that represents it in the bicomplex. The corresponding generator on $^{v}E^1_{p,q}$ will be in a higher (or equal if the staircase has length 0) vertical grading than the generator in $^{h}E^1_{p,q}$. This gives the desired inequality since the vertical grading on $^{v}E^1_{p,q}$ gives the grading on $H_*(X)$, and the vertical grading on $^{h}E^1_{p,q}$ gives the grading on $H_*(F)$.
\end{proof}

\section{Knot Floer homology background}\label{sec:3}
In this section we will prove Theorem \ref{thm:hls}, and recall some other useful properties of knot Floer homology. Throughout the rest of the paper, let $\widetilde{K}$ be a 2-periodic knot with axis $\widetilde{A}$, and let $K$ be the quotient knot with axis $A$. Let $\lambda$ be the linking number of $K$ with $A$, $i$ be the Maslov grading and $a$ be the Alexander grading. We now prove Theorem \ref{thm:hls} from \cite[Theorem 1.16]{HLS}.

 \begin{thm:hls}
 There is a rank inequality
\[
\sum_i \rank \bigg{(} \widehat{\mathit{HFK}}_i(\widetilde{K}, 2a + \dfrac{\lambda -1}{2}) \oplus  \widehat{\mathit{HFK}}_i(\widetilde{K}, 2a + \dfrac{\lambda +1}{2}) \bigg{)}\geq \sum_i \rank \widehat{\mathit{HFK}}_i(K, a)
\]
for all $i,a \in \Z$. 
 \end{thm:hls}

 \begin{proof}
 Let $V$ and $W$ be 2-dimensional vector spaces with gradings as shown in Figure \ref{fig:gradings}. 
 
 \begin{figure}

\begin{center}
\begin{tabular}{ |c||c|c| } 

\multicolumn{3}{l}{$V:$} \\
 
 \hline
 gr($V$) & $0$ & $1$ \\ 
 \hline
 \hline
 $-1$ & $\F_2$ & $0$ \\ 
 \hline
 $0$ & $0$ & $\F_2$ \\ 
 \hline
\end{tabular}
\quad \quad \quad
\begin{tabular}{ |c||c|c| } 

\multicolumn{3}{l}{$W:$} \\
 
 \hline
 gr($W$) & $-1$ & $0$ \\ 
 \hline
 \hline
 $0$ & $\F_2$ & $\F_2$ \\ 
 \hline
\end{tabular}
\end{center}

\caption{The 2-dimensional vector spaces $V$ and $W$. Columns are Maslov gradings, and rows are Alexander gradings.}
\label{fig:gradings}
\end{figure}

 Then \cite[Theorem 1.16]{HLS} provides a spectral sequence
 \begin{equation}\label{eqn:ss}
 \widehat{\mathit{HFK}}_*(\widetilde{K}) \otimes V \otimes W \otimes \F_2[\theta, \theta^{-1}] \Rightarrow \widehat{\mathit{HFK}}_*(K) \otimes W \otimes \F_2[\theta, \theta^{-1}].
 \end{equation}
which splits along Alexander gradings, taking the grading $2a + \dfrac{\lambda - 1}{2}$ on the $E^1$ page to $a$ on the $E^{\infty}$ page, and gradings of the other parity on the $E^1$ page to 0 on the $E^{\infty}$ page. In particular, the factors $V$ and $W$ are essential to the spectral sequence. Notice that tensoring a complex $X$ with $V$ is the same as taking a direct sum of $X$ with itself after shifting the gradings.

Consider the grading $\widetilde{a} = 2a + \dfrac{\lambda -1}{2}$ on the $E^1$ page. Then there are exactly two gradings ($\widetilde{a}$ and $\widetilde{a} +1$) in $\widehat{\mathit{HFK}}(\widetilde{K})$ which contribute to that $\widetilde{a}$ grading in the tensor product. Furthermore, these two gradings do not contribute to any other gradings in the tensor product. Hence the spectral sequence (\ref{eqn:ss}) gives the result.
 \end{proof}
 
The following theorems of Ozsv\'ath and Szab\'o characterize knot Floer homology for alternating knots and L-space knots respectively in such a way that they can be recovered from the Alexander polynomial. These will be useful in obtaining applications of Conjecture \ref{conj:ineq}.
 \begin{theorem}\cite[Theorem 1.3]{OS}
\label{thm:alt}
Let $K \subset S^3$ be an alternating knot, and write its (symmetrized) Alexander polynomial as
\[
\Delta_K(t) = a_0 + \sum_{s>0}a_s(t^s + t^{-s}).
\]
Then $\widehat{\mathit{HFK}}(S^3,K,s)$ is supported entirely in homological degree $s + \sigma(K)/2$, and 
\[
\widehat{\mathit{HFK}}(S^3,K,s) \cong \Z^{|a_s|}.
\]
\end{theorem}

\begin{theorem}\cite[Theorem 1.2]{OS2}
\label{thm:lspace}
Let $K \subset S^3$ be an L-space knot. Then there is an increasing sequence of integers
\[
n_{-k} < \dots < n_k
\]
with $n_i = -n_{-i}$, such that for $-k \leq i \leq k$ and
\[
\delta_i = 
\begin{cases} 
      0 & \mbox{ if }i = k \\
      \delta_{i+1} - 2(n_{i+1} - n_i) + 1 & \mbox{ if }k-i \mbox{ is odd}\\
      \delta_{i+1} - 1 &\mbox{ if }k-i>0 \mbox{ is even}, 
   \end{cases}
\]
$\widehat{\mathit{HFK}}(K,a) = 0$ unless $a = n_i$ for some $i$. In this case $\widehat{\mathit{HFK}}(K,a) \cong \Z$ and is supported entirely in homological degree $\delta_i$.
\end{theorem}
\section{Consequences of a filtered rank inequality}\label{sec:4}
The goal of this section is to prove some interesting consequences of Conjecture \ref{conj:ineq}. Specifically, we will prove some restrictions on the Alexander polynomials of certain periodic knots. To begin, we restate the conjecture.

\begin{conj:ineq} 
Let $\widetilde{K} \in S^3$ be $2$-periodic with quotient knot $K$. Then for all $a,q \in \Z$,
\begin{align*}
\sum_{i \geq q} \rank \bigg{(}  \widehat{\mathit{HFK}}_i(\widetilde{K},\widetilde{a}) \oplus \widehat{\mathit{HFK}}_i(\widetilde{K},\widetilde{a} + 1) \bigg{)} &\geq  \sum_{2i \geq q+1} \rank \widehat{\mathit{HFK}}_i(K,a)
\\
\intertext{and}
\sum_{i \leq q} \rank \bigg{(} \widehat{\mathit{HFK}}_i(\widetilde{K},-\widetilde{a}) \oplus \widehat{\mathit{HFK}}_i(\widetilde{K},-\widetilde{a} - 1) \bigg{)} &\geq  \sum_{2i \leq q-1} \rank \widehat{\mathit{HFK}}_i(K,-a),
\end{align*}
where $\widetilde{a} = 2a + \dfrac{\lambda-1}{2}$.
\end{conj:ineq}

Theorem \ref{thm:hls} and this conjecture both have some nice consequences for the Alexander polynomials of 2-periodic alternating and L-space knots. These follow from the theorems of Ozsv\'{a}th and Szab\'{o} stated in the previous section.

\begin{theorem}
Let $\widetilde{K}$ be a $2$-periodic alternating knot in $S^3$ with alternating quotient $K$ and having linking number $\lambda$ with the axis. Let the Alexander polynomials of $\widetilde{K}$ and $K$ be
\[
\Delta_{\widetilde{K}}(t) = \widetilde{a}_0 + \sum_{\widetilde{s} > 0} \widetilde{a}_{\widetilde{s}} (t^{\widetilde{s}} + t^{-\widetilde{s}}), \mbox{ and } \Delta_{K(t)} = a_0 + \sum_{s > 0} a_s (t^s + t^{-s}),
\]
respectively. Then for each $s$,
\[
|\widetilde{a}_{2s + \frac{\lambda -1}{2}} - \widetilde{a}_{2s + \frac{\lambda +1}{2}}| \geq a_s,
\]
and in particular the number of terms in $\Delta_{\widetilde{K}}$ is at least the number of terms in $\Delta_K$. Additionally, if Conjecture \ref{conj:ineq} holds then 
\[
|2\sigma(K) - \sigma(\widetilde{K})| \leq \lambda + 1.
\]
\label{thm:altineq}
\end{theorem} 
\begin{proof}
The statement follows directly from applying the two inequalities in Conjecture \ref{conj:ineq} to Theorem \ref{thm:alt}. In particular since the inequality is split into Alexander gradings, we can consider $\Delta_K$ one term at a time. Then the inequality $|\widetilde{a}_{2a + \frac{\lambda -1}{2}} - \widetilde{a}_{2a + \frac{\lambda +1}{2}}| \geq a_s$ comes from the total rank inequality in Theorem \ref{thm:hls}, noting that signs on the coefficients of $\widetilde{K}$ alternate. The grading refinement immediately gives
\[
\widetilde{s} + \dfrac{\sigma(\widetilde{K})}{2} \geq 2s + \sigma(K) - 1,
\]
for each grading $\widetilde{s}$ sent to $s$ by the spectral sequence (\ref{eqn:ss}). However, we know that the gradings $2s + \frac{\lambda -1}{2}$ and $2s + \frac{\lambda +1}{2}$ get sent to $s$, so this simplifies to
\[
2\sigma(K) + \lambda +1 \geq \sigma(\widetilde{K}).
\]
Finally, by considering the mirror of $K$ we also get that 
\[
\sigma(\widetilde{K}) \geq 2 \sigma(K) -\lambda -1,
\]
as desired.
\end{proof}
\begin{remark}
For odd order periodic alternating knots, the quotient is automatically alternating, and this has been conjectured for 2-periodic knots as well, in which case the assumption that the quotient is alternating may be removed from Theorem \ref{thm:altineq}. See \cite{B}.
\end{remark}
\begin{example}
Consider the knot $10_{122}$ which is 2-periodic over $4_1$ with $\lambda = 1$. $10_{122}$ has signature $0$ and Alexander polynomial
\[
2t^{-3} + 11t^2-24t+31-24t^{-1}+11t^{-2}-2t^{-3},
\]
whereas $4_1$ also has signature $0$, but Alexander polynomial
\[
-t +3 - t^{-1}.
\]
Looking back at Theorem \ref{thm:alt}, we have Alexander gradings given by the exponents in $\Delta_K$ so that $s \in \{-1,0,1\}$ with $a_s \in \{1,3,1\}$ respectively. Since $\lambda = 1$ these will lift to give $\widetilde{s} = 2s + 0$, and indeed the first inequality is then $2 +11 \geq 1$, $24+31 \geq 3$, and $24+11 \geq 1$. The signature inequality is also satisfied with $2 \geq 1$, $0 \geq -1$, and $-2 \geq -3$. For the $\overline{s}$ inequalities, the computation is similar.
\end{example}
\begin{remark}
The fact that the number of terms in $\Delta_{\widetilde{K}}$ is at least the number of terms in $\Delta_{K}$ also follows from a theorem of Murasugi that all terms in the Alexander polynomial of an alternating knot are nonzero \cite[Theorem 1.1]{Mu2}.
\end{remark}

\begin{theorem}
Let $\widetilde{K}$ be a $2$-periodic L-space knot in $S^3$ with L-space quotient $K$. Then there are at least as many terms in $\Delta_{\widetilde{K}}$ as in $\Delta_K$.  Furthermore let $n$ be the width of $\Delta_K$, again normalize the Alexander polynomial as in Theorem \ref{thm:alt}, and suppose that Conjecture \ref{conj:ineq} holds. Then there is at most one term in $\Delta_{\widetilde{K}}$ with exponent larger than 
\[
2n + \dfrac{\lambda + 1}{2},
\]
and in particular there are at most $4n + \lambda + 4$ terms in $\Delta_{\widetilde{K}}$ total.
\label{thm:lspaceineq}
\end{theorem}
\begin{proof}
As we will see, all statements follow from Theorem \ref{thm:lspace}, the characterization of $\widehat{\mathit{HFK}}(K)$ in terms of $\Delta_K$. 

The inequality between the number of terms in $\Delta_{\widetilde{K}}$ and $\Delta_K$ is clear from Theorem \ref{thm:hls}.

For the other claims, observe that the largest $\delta_i$ in Theorem \ref{thm:lspace} is zero, so that on the maximal Maslov grading Conjecture \ref{conj:ineq} will be trivially satisfied. The other conclusions will follow by considering the minimal Maslov grading. Observe that the smallest $\delta_i$ is negative the width of the Alexander polynomial, $n_{-k} - n_k$, as follows. Since the Alexander polynomial is symmetric each gap $n_{i+1}-n_i$ has a mirrored gap $n_{-i}-n_{-i-1}$, and exactly one of these contributes $2(n_{i+1} - n_i) +1$, while the other contributes $-1$. Summing these gives that indeed the minimal $\delta_i$ is $n_{-k}-n_k$.

This gives the stated bound on the number of terms in $\Delta_{\widetilde{K}}$ of degree larger than $2n + (\lambda +1)/2$ since otherwise the $\delta_i$ for $\widetilde{K}$ corresponding to the minimal $\delta_i$ for $K$ would be too negative.

Finally, the bound on the number of terms in $\Delta_{\widetilde{K}}$ follows from symmetry. Specifically there is also at most one term in $\Delta_{\widetilde{K}}$ with exponent less than $-2n - (\lambda+1)/2$, and hence there are at most $4n + \lambda + 4$ terms total.
\end{proof}

This theorem can be somewhat improved by further assuming the L-space conjecture of Boyer, Gordon and Watson.
\begin{conjecture}\cite[Conjecture 1]{BGW}
Let $M$ be a closed, connected, irreducible, orientable $3$-manifold. Then $M$ is not an L-space if and only if $\pi_1(M)$ is left-orderable.
\label{conj:bgw}
\end{conjecture}
In particular, assuming this conjecture allows us to drop the assumption that $K$ is an L-space knot in Theorem \ref{thm:lspaceineq}.
\begin{proposition}
Let $\widetilde{K}$ be a $p$-periodic knot with quotient $K$. If Conjecture \ref{conj:bgw} holds and $\widetilde{K}$ is an L-space knot, then $K$ is an L-space knot.
\end{proposition}
\begin{proof}
Since $\widetilde{K}$ is an L-space knot, all sufficiently large surgeries on $\widetilde{K}$ are L-spaces. In particular, by taking any large surgery with surgery coefficient a multiple of $p$, we get an L-space surgery $\widetilde{Y} = S^3_{pn}(\widetilde{K})$ with a surgery curve that is equivariant with respect to the periodic action. This then induces a surgery on the quotient knot $Y = S^3_n(K)$. Furthermore, $\widetilde{Y}$ is a $p$-fold branched cover of $Y$ with branch set the union of the core of the surgery and the axis of the original periodic action. We can also assume that $\widetilde{Y}$ and $Y$ are irreducible, since there are only finitely many reducible surgeries on a given knot. 

Now we claim that if $\pi_1(Y)$ is left-orderable, then so is $\pi_1(\widetilde{Y})$. This follows directly from \cite[Theorem 1.1(1)]{BRW} if the induced map $\pi_1(\widetilde{Y}) \to \pi_1(Y)$ is non-trivial. Suppose that the map is trivial. Then we can lift the map $\widetilde{Y} \to Y$ to the universal cover $\overline{Y}$ of $Y$.  If $\overline{Y}$ is not $S^3$, then $H_3(\overline{Y}) = 0$, and so the map $\widetilde{Y} \to Y$ has degree 0, contradicting it being a $p$-fold branched cover. On the other hand, if $\overline{Y}$ is $S^3$, then $\pi_1(Y)$ is finite and hence not left-orderable. 

Now Conjecture \ref{conj:bgw} implies that if $\widetilde{Y}$ is an L-space then so is $Y$. 
\end{proof}

\section{Evidence for the main conjecture}\label{sec:5}
There is strong evidence for Conjecture \ref{conj:ineq}, both theoretically and computationally. 

\subsection{Computational Evidence}
To check Conjecture \ref{conj:ineq}, we generated pseudo-random knots and verified the conjecture for each one as follows.

First we construct a tangle $K$ on 5 strands by choosing 18 random operations from the set $\{c_i, o_i, u_i\}$. Here $c_i$ refers to a cup cap pair connecting the $i$th strand to the $(i+1)$th strand, $o_i$ refers to the $i$th strand crossing over the $(i+1)$th strand, and $u_i$ refers to crossing the $i$th strand under the $(i+1)$th strand. 

Next, we check that each $K$ we construct has closure a knot, and that the tangle for $\widetilde{K}$ constructed by repeating the operations for $K$ also has closure a knot. If either condition fails, then we choose 18 new random operations. 

Once we have a 2-periodic knot described by a tangle, we use Ozsv\'{a}th and Szab\'{o}'s knot Floer homology calculator \cite{HFKcalc} based on \cite{OS3} to compute $\widehat{\mathit{HFK}}(K)$ and $\widehat{\mathit{HFK}}(\widetilde{K})$, and verify Conjecture \ref{conj:ineq} for this pseudo-random 2-periodic knot. 

While verifying the conjecture for each knot, we also tabulated the Alexander polynomial and the total rank of the knot Floer homology for each periodic knot. The total rank of $\widehat{\mathit{HFK}}(\widetilde{K})$ ranged from 1 to 907253 with an average of about 7761.52. These data confirm that we have verified the conjecture for over 500 distinct knots. 

Additionally, we note that the signature inequality $|2\sigma(K) - \sigma(\widetilde{K})| \leq \lambda + 1$ holds for these knots even if the knot is not alternating, leading to the following conjecture.
\begin{conjecture}
Let $\widetilde{K}$ be a 2-periodic knot with quotient knot $K$, and let $\lambda$ be the linking number between $K$ and the axis. Then
\[
|2\sigma(K) - \sigma(\widetilde{K})| \leq \lambda + 1.
\]
\end{conjecture}

\subsection{The case of torus knots}

It does not seem easy to check many special cases of Theorem \ref{thm:hls} or Conjecture \ref{conj:ineq}. For torus knots, specific examples may be computed by Theorem \ref{thm:lspace}, which we have done for many torus knots.
\begin{proposition}
Conjecture \ref{conj:ineq} is true for $\widetilde{K} = T(2p,q)$ and $K = T(p,q)$ for all $p,q < 60$.
\end{proposition} 
\begin{proof}
Since torus knots have an explicit formula for their Alexander polynomials, and are L-space knots, we used a computer to directly compute $\widehat{\mathit{HFK}}$ using Theorem \ref{thm:lspace}.
\end{proof}
On the other hand, computations for any infinite family involve understanding all terms in some cyclotomic polynomials. Nonetheless, we can check the main conjecture in this case if we restrict to only the maximal Alexander gradings, and we can verify the results of Theorem \ref{thm:lspaceineq} for torus knots even without assuming the conclusion of Conjecture \ref{conj:ineq}.

\begin{proposition}
The first inequality in Conjecture \ref{conj:ineq} is true for the maximal Alexander gradings on the $2$-periodic torus knots $T(2p,q) \to T(p,q)$.
\end{proposition}
\begin{proof}
Since torus knots have L-space surgeries, we can use Theorem \ref{thm:lspace} to compute $\widehat{\mathit{HFK}}$. Recall that 
\[
\Delta_{T(p,q)}(t) = \dfrac{(t^{pq}-1)(t-1)}{(t^p-1)(t^q-1)}
\]
has degree $(p-1)(q-1)$, and that in this case the linking number between the axis and knot is $\lambda = q$. By Theorem \ref{thm:lspace}, the maximum Alexander grading for $T(p,q)$ is $(p-1)(q-1)/2$, half the width of $\Delta_{T(p,q)}$, which lifts to the Alexander grading
\[
(p-1)(q-1) +\dfrac{q-1}{2} =  \dfrac{2pq-2p-q+1}{2} = \dfrac{(2p-1)(q-1)}{2}.
\]
Conveniently, this is the maximum Alexander grading for $\Delta_{T(2p,q)}$. And indeed, these Alexander polynomials are monic, and both the $\delta_i$'s from Theorem \ref{thm:lspace} are 0, giving the desired result.
\end{proof}
\begin{remark}
The above proposition is also true, with essentially the same proof, for the mirror knots, or equivalently for the minimum Alexander grading in the second inequality in Conjecture \ref{conj:ineq}.
\end{remark}

\begin{proposition}
The conclusions of Theorem \ref{thm:lspaceineq} hold for torus knots, without assuming Conjecture \ref{conj:ineq}.
\end{proposition}
\begin{proof}
This follows immediately by checking the degrees of the Alexander polynomials for torus knots. As in the previous proposition, we see that there are no terms in $\Delta_{T(2p,q)}$ larger than $2\cdot$width$(\Delta_{T(p,q)}) + (q+1)/2$.
\end{proof}

\subsection{Adapting the Morse homology proof}\label{sec:5.3}
Finally, we would like to point out how a naive attempt at adapting the proof of Theorem \ref{thm:floyd} to prove Conjecture \ref{conj:ineq} fails. In fact, most of the proof works similarly.
\begin{proposition}
If the spectral sequence (\ref{eqn:ss}) does not contain any staircases beginning with a vertical differential on the top left and ending with a horizontal differential on the bottom right, then Conjecture \ref{conj:ineq} holds.
\end{proposition}
\begin{proof}
This condition is a slightly weaker replacement of Lemma \ref{lemma:sub}. From there, the proof follows identically to that of Theorem \ref{thm:floyd}. The factor of 2 in the grading shift comes from the identification of the $E^{\infty}$ page with $\widehat{\mathit{HFK}}_*(K)\otimes W \otimes \F_2[\theta,\theta^{-1}]$ as in \cite{HLS}. The shift by 1 in the grading comes from the extra $V$ vector space in the spectral sequence.
\end{proof}

\bibliography{bibliography}{}
\bibliographystyle{alpha}

\end{document}